\documentclass[a4paper, english]{article}

\usepackage{mathrsfs}     
\usepackage{bussproofs}	  

\usepackage{graphicx}
\usepackage{a4wide}
\usepackage{authblk}
\usepackage{amsthm}

\usepackage{amsthm}
\usepackage{lscape}
\usepackage{pdflscape}

\usepackage{stmaryrd}

\usepackage{cmll}

\usepackage{comment}

\usepackage{amscd}
\usepackage{amssymb,amsmath}
\usepackage{mathptmx}
\usepackage{mathrsfs}
\usepackage{color}
\usepackage{xspace}
\usepackage{bussproofs}
\EnableBpAbbreviations

\usepackage{tikz}

  \usepackage{txfonts}

\usepackage{bpextra}

\usepackage{enumerate}

\usepackage{centernot}

\usepackage{mathtools}

\usepackage{hyperref}

\usepackage{cleveref}

\usepackage{relsize}

\usepackage{caption}

\usepackage{multirow}

\usepackage{multicol}

\usepackage{array}
\newcolumntype{C}[1]{>{\centering\arraybackslash}p{#1}}
\newcolumntype{L}[1]{>{\arraybackslash}p{#1}}

\def\fCenter{{\mbox{$\ \vdash\ $}}}

\newcommand{\fns}{\footnotesize}

\newcommand{\msf}{\mathsf}

\newcommand{\imp}{\rightarrow}

\newcommand{\D}{\Diamond}
\newcommand{\B}{\Box}

\usepackage{commath}
\usepackage{bm}
\usepackage{enumitem}

\newcommand{\mbf}{\mathbf}

\newcommand{\mbb}{\mathbb}

\usetikzlibrary{arrows}
\usetikzlibrary{matrix}
\usetikzlibrary{patterns}
\usetikzlibrary{shapes}

\theoremstyle{plain}
\newtheorem{theorem}{Theorem}[section]

\newtheorem*{theorem*}{Theorem}
\newtheorem*{lemma*}{Lemma}

\theoremstyle{definition}

\newtheorem*{definition*}{Definition}
\newtheorem*{proposition*}{Proposition}
\newtheorem*{corollary*}{Corollary}
\newtheorem*{fact*}{Fact}
\newtheorem*{conjecture*}{Conjecture}

\theoremstyle{remark}

\newtheorem*{remark*}{Remark}
\newtheorem*{example*}{Example}
\newtheorem*{claim*}{Claim}

\theoremstyle{definition}
\newtheorem{thm}[theorem]{Theorem}
\newtheorem{defn}[theorem]{Definition}
\newtheorem{lem}[theorem]{Lemma}
\newtheorem{cor}[theorem]{Corollary}
\newtheorem{prop}[theorem]{Proposition}

\newtheorem{exam}[theorem]{Example}
\newcommand{\den}[1]{\llbracket{#1}\rrbracket}

\newcommand{\weg}[1]{}

\title{Neighbourhood semantics for graded modal logic}
\author[1]{Jinsheng Chen\thanks{Jinsheng Chen is supported by the National Social Science Foundation of China under Grant No. 20\&ZD047.}}
\author[2]{Hans van Ditmarsch}
\author[3]{Giuseppe Greco\thanks{Giuseppe Greco is supported by the NWO grant KIVI.2019.001.}}

\author[3]{Apostolos Tzimoulis}
\affil[1]{Zhejiang University, Department of Philosophy}
\affil[2]{LORIA, CNRS, University of Lorraine}
\affil[3]{Vrije Universiteit Amsterdam, School of Business and Economics, Ethics, Governance and Society}
\date{}
\begin{document}
\maketitle
	\begin{abstract}
	We introduce a class of neighbourhood frames for graded modal logic embedding Kripke frames into neighbourhood frames. This class of neighbourhood frames is shown to be first-order definable but not modally definable. We also obtain a new definition of graded bisimulation with respect to Kripke frames by modifying the definition of monotonic bisimulation. 
		
		
	\end{abstract}


\section{Introduction}

Graded modal logic $\mbf{GrK}$ is an extension of propositional logic with graded modalities $\D_n (n \in \mbb{N})$ that count the number of successors of a given state. The interpretation of formula $\D_n \varphi$ in a Kripke model is that the number of successors that satisfy $\varphi$ is at least $n$.  
Originally introduced in Goble \cite{goble1970grades}, the notion of a graded modality is developed so that `propositions can be distinguished by degrees or grades of necessity or possibility' \cite[Page 1]{goble1970grades}. This language was studied in Kaplan \cite{kaplan1970s5} as an extension of $\mbf{S5}$. Fine \cite{fine1972so}, De Caro \cite{de1988graded} and Cerrato \cite{cerrato1990general} investigated the completeness of $\mbf{GrK}$ and its extensions. 
Van der Hoek \cite{van1992semantics} investigated the expressibility, decidability and definability of graded modal logic and also correspondence theory. Cerrato \cite{cerrato1994decidability} proved the decidability by filtration for graded modal logic.

De Rijke \cite{de2000note} introduced graded tuple bisimulation for graded modal logic. Using this he proved the finite model property  (which was first proved in Cerrato \cite{cerrato1994decidability} via filtration) and that a first-order formula is invariant under graded bisimulation iff it is equivalent to a graded modal formula. Aceto, Ingolfsdottir and Sack \cite{aceto2010resource} showed that resource bisimulation and graded bisimulation coincide over image-finite Kripke frames. Van der Hoek and Meyer \cite{van1992graded} proposed a graded modal logic $\mbf{GrS5}$, which is seen as a graded epistemic logic and is able to express `accepting $\varphi$ if there are at most $n$ exceptions to $\varphi$'. Ma and van Ditmarsch \cite{ma2016dynamic} developed dynamic extensions of graded epistemic logics.

Monotonic modal logics are weakenings of normal modal logics in which the \emph{additivity} ($\D \bot \leftrightarrow \bot$ and $\D p \lor \D q \leftrightarrow \D ( p\lor q)$) of the diamond modality has been weakened to \emph{monotonicity} ($\D p \lor \D q \leftrightarrow \D( p \lor q)$), which can also be formulated as a derivation rule: from $\vdash \varphi \to \psi$ infer $\vdash \D \varphi \to \D \psi$. Monotonic modal logics are interpreted over monotonic neighbourhood frames, that is neighbourhood frames where the collection of neighbourhoods of a point is closed under supersets. There have been many results about monotonic modal logics and monotonic neighbourhood frames \cite{chellas1980modal,hansen2003monotonic,pacuit2017neighborhood},  including model constructions, definability, correspondence theory, canonical model constructions, algebraic duality, coalgebraic semantics, interpolation, simulations of monotonic modal logics by bimodal normal logics, etc.

In this paper, we propose a neighbourhood semantics for graded modal logic. We define an operation $(.)^\bullet$ (Def.~\ref{defn:bullet}) to obtain a class of monotonic neighbourhood frames on which graded modal logic is interpreted. This class of neighbourhood frames is shown to be first-order definable in Section \ref{sec:FOdef} and modally undefinable in Section \ref{sec:modalundef}. In Section \ref{sec:bisimulation} we obtain a new definition of graded bisimulation with respect to Kripke frames by modifying the definition of monotonic bisimulation and show that it is equivalent to the one proposed in \cite{de2000note}. Our results show that techniques for monotonic modal logics can be successfully applied to graded modal logic. 


\section{Preliminaries}\label{sec:2}
\subsection{Graded modal logic}

{\bf Language}. \ \  Let $\msf{Prop}$ be a set of proposition letters. Language $\mathcal{L}_g$ is defined by induction as follows:
\[
\mathcal{L}_g \ni \varphi ::= p \mid \neg \varphi \mid (\varphi \lor \varphi) \mid \D_n \varphi
\]
where $p \in \msf{Prop}$ and $n \in \mathbb{N}$. We recall that $\mathbb{N}$ is the set of natural numbers. The \emph{complexity} of a formula $\varphi \in \mathcal{L}_g $ is the number of connectives occurring in $\varphi$. 
Other propositional connectives $\bot$, $\top$, $\land$, $\to$, $\leftrightarrow$ are defined as usual. The \emph{dual} of $\D_n \varphi$ is defined as $\B_n \varphi: = \neg \D_n \neg \varphi$. Further, define $\D \varphi : =\D_1 \varphi$ and $\D_{!n} \varphi:= \D_n \varphi \land \neg \D_{n+1} \varphi$. The interpretation of a formula $\D_n \varphi$ in a Kripke model is that the number of successors that satisfy $\varphi$ is at least $n$. The interpretation of formula $\D_{!n} \varphi$ is that the number of successors that satisfy $\varphi$ is exactly $n$.

\noindent {\bf Kripke semantics.} \ \ A \emph{Kripke frame} is a pair $ (W,R)$, denoted $\mathcal{F}$, where $W$ is a set of states and $R$ is a binary relation on $W$. Denote by $\msf{F}_K$ the class of all Kripke frames. A \emph{Kripke model} is a pair $\mathcal{M}=(\mathcal{F},V)$ where $\mathcal{F}$ is a Kripke frame and $V: \msf{Prop} \to \mathcal{P}(W)$ is a \emph{valuation}. For model $\mathcal{M}=(W,R,V)$ and $w \in W$, we call $\mathcal{M}, w$ a \emph{pointed model}.

Given a set $X$, denote by $\abs{X}$ the \emph{cardinality} of $X$. Suppose that $w$ is a state in a Kripke model $\mathcal{M} = (W,R,V)$. The \emph{truth} of a $\mathcal{L}_g$-formula $\varphi$ at $w$ in $\mathcal{M}$, notation $\mathcal{M}, w \Vdash \varphi$, is defined inductively as follows:
\[
\begin{array}{lcl}
\mathcal{M}, w \Vdash p &\text{~iff~}&\qquad  p \in V(p) \\
\mathcal{M}, w \Vdash \neg \psi &\text{~iff~}&\qquad \mathcal{M}, w \not \Vdash \psi\\
\mathcal{M}, w \Vdash \psi_1 \lor \psi_2  &\text{~iff~}&\qquad\mathcal{M}, w \Vdash \psi_1 \text{~or~} \mathcal{M}, w \Vdash \psi_2 \\
\mathcal{M}, w \Vdash \D_n \psi  &\text{~iff~}& \qquad\abs{ R[w] \cap \llbracket \psi \rrbracket_{\mathcal{M}} }\ge n
\end{array}
\]
where $R[w] = \{ v \in W : Rwv\}$ is the set of $w$-successors and  $\llbracket \psi \rrbracket_{\mathcal{M}}= \{v \in W: \mathcal{M} ,v \Vdash \psi\}$ is the \emph{truth set} of $\varphi$ in $\mathcal{M}$. For a set $\Gamma$ of $\mathcal{L}_g$-formulas, we write $\mathcal{M}, w \Vdash \Gamma$ if $\mathcal{M}, w \Vdash \varphi$ for all $\varphi \in \Gamma$. Pointed models $\mathcal{M},w$ and $\mathcal{M}',w'$ are said to be \emph{modally equivalent} (notation: $\mathcal{M},w \equiv_k \mathcal{M}',w'$)  if for all $\mathcal{L}_g$-formulas $\varphi$, we have $\mathcal{M},w\Vdash \varphi $ iff $\mathcal{M}',w'\Vdash \varphi$.

A formula $\varphi$ is \emph{valid at a state $w$ in a frame $\mathcal{F}$}, notation $\mathcal{F},w \Vdash \varphi$, if $\varphi$ is true at $w$ in every model $(\mathcal{F},V)$ based on $\mathcal{F}$; $\varphi$ is \emph{valid in a frame $\mathcal{F}$}, notation $\mathcal{F} \Vdash \varphi$, if it is valid at every state in $\mathcal{F}$; $\varphi$ is  \emph{valid in a class of frames $S_K$}, notation $\Vdash_{S_K} \varphi$, if $\mathcal{F}\Vdash \varphi$ for all $\mathcal{F}\in S_K$.

Let $S_K$ be a class of Kripke frames and $\Gamma \cup\{\varphi\}$ a set of $\mathcal{L}_g$-formulas. We say that $\varphi$ is a \emph{(local) semantic consequence of $\Gamma$ over $S_K$}, notation $\Gamma \Vdash_{S_K} \varphi$, if for all models $\mathcal{M}$ based on frames in $S_K$, and all states in $\mathcal{M}$, if $\mathcal{M}, w\Vdash \Gamma$ then $\mathcal{M}, w \Vdash \varphi$.

\medskip

\noindent {\bf Graded semantics}. \ \ In this subsection, we recall the graded semantics from Ma and van Ditmarsch \cite{ma2016dynamic}. The sum operation and the `greater than or equal to' relation $(\ge)$ are defined over natural numbers $\mathbb{N}$ plus $\omega$, the least ordinal number greater than any natural number, i.e., $\forall n\in \mathbb{N},n < \omega$. Variables $n,m,i ,j$ range over the natural numbers $\mbb{N}$, not over $\mathbb{N}\cup\{\omega\}$.

A \emph{graded frame} is a pair $\mathfrak{f} = (W,\sigma)$, where $W$ is a set of states and $\sigma:W \times W \to \mbb{N}\cup\{\omega\}$ is a function assigning a natural number or $\omega$ to each pair of states. Denote by $\msf{F}_G$ the class of all graded frames. A \emph{graded model} is a pair $\mathfrak{M} = (\mathfrak{f} ,V)$ where $\mathfrak{f}$ is a graded frame and $V: \msf{Prop} \to \mathcal{P}(W)$ is a valuation. 

For $X\subseteq W$ and $w \in W$, define $\sigma (w,X) $ as $\Sigma_{u \in X} \sigma(w,u)$, the sum of $\sigma(w,u)$ for all $u \in X$. In particular, we define $\sigma(w,\emptyset) = 0$. The notation $X\subseteq_{<\omega} W$ represents that $X$ is a finite subset of $W$ and $\mathcal{P}_{<\omega}(W)$ is the set of finite subsets of $W$.

Suppose that $w$ is a state in a graded model $\mathfrak{M} = (W,\sigma,V)$. The \emph{truth} of a $\mathcal{L}_g$-formula $\varphi$ at $w$ in $\mathfrak{M}$, notation $\mathfrak{M}, w \Vdash \varphi$, is defined inductively as follows:
\[
\begin{array}{lcl}
\mathfrak{M},w \Vdash p &\text{~iff~}&\quad \quad w \in V(p)\\
\mathfrak{M},w \Vdash \neg \psi &\text{~iff~}&\quad\quad \mathfrak{M},w \not\Vdash \psi\\
\mathfrak{M},w \Vdash \psi_1 \lor \psi_2 &\text{~iff~}&\quad \quad \mathfrak{M},w \Vdash \psi_1 \text{~or~} \mathfrak{M},w \Vdash \psi_2\\
\mathfrak{M},w \Vdash \D_n \psi &\text{~iff~}&\quad\quad \exists X \subseteq_{<\omega} W \ (\sigma(w,X) \ge n \ \& \ X \subseteq \llbracket \psi \rrbracket_{\mathfrak{M}})
\end{array}
\]

To our knowledge, graded frames first appeared in \cite{de1988graded} as an intermediate structure to prove completeness of $\mbf{GrK}$ with respect to Kripke frames. They are called \emph{multiframes} in \cite{aceto2010resource}. Graded frames are alternative semantics for graded modal logic, indeed each graded frame can be associated with a Kripke frame validating the same formulas, and vice versa as follows (cf. \cite[Proposition 2.12 ]{ma2016dynamic}): Given a Kripke frame $\mathcal{F}=(W,R)$, the associated graded frame $\mathcal{F}^\circ =(W,\sigma)$ is defined by setting $\sigma(w,u)= 1$ if $wRu$, and $\sigma(w,u)= 0$ otherwise; given a graded frame $\mathcal{F}=(W, \sigma)$, the associated Kripke frame $\mathcal{F}_\circ= (W_\circ,R)$ is defined by setting $W_\circ =\{(w,i) \mid w \in W \ \& \ i\in \mathbb{N}\cup\{\omega\}\}$ and $(w,i) R (u,j)$ iff $\sigma(w,u)\ge j >0$.

\medskip

\noindent {\bf Axiomatization}. \ \ The \emph{minimal graded modal logic} $\mbf{GrK}$ consists of the following axiom schemas and inference rules:
\begin{align*}
&(Ax1)~~ \text{all instances of propositional tautologies}\\
&(Ax2)~~ \D_0 \varphi \leftrightarrow \top\\
&(Ax3)~~ \D_n \bot \leftrightarrow \bot \qquad(n > 0)\\
&(Ax4)~~ \D_{n+1} \varphi \to  \D_n \varphi\\
&(Ax5)~~ \Box(\varphi \to \psi ) \to( \D_n \varphi \to  \D_n \psi)\\
&(Ax6)~~ \neg \D (\varphi \land \psi)\land  \D_{!m} \varphi \land  \D_{!n}\psi\to  \D_{!(m+n)}(\varphi \lor \psi)\\
&(MP)~~ \text{from~} \varphi \text{~and ~}\varphi \to\psi \text{~infer~} \psi\\
&(Gen)~~ \text{from~} \varphi \text{~infer~} \B \varphi
\end{align*}

The set of theorems derivable in the system $\mbf{GrK}$ is also called $\mbf{GrK}$. A \emph{graded modal logic} is a set $\Lambda$ of $\mathcal{L}_g$-formulas with $\mbf{Grk} \subseteq \Lambda$. \weg{For $\varphi \in \Lambda$, we write $\vdash_\Lambda \varphi$. If $\Gamma$ is a set of axioms, then $\mbf{GrK\Gamma}$ is the logic obtained by adding axioms in $\Gamma$ as axiom schemas to $\mbf{GrK}$.} If $\varphi \in \Lambda$, we write $\vdash_\Lambda \varphi$.

\begin{thm}[\cite{de1988graded}]\label{thm:2106}
$\mbf{GrK}$ is sound and complete with respect to the class of all Kripke frames. 
\end{thm}

\begin{thm}[Theorem 3.2 of \cite{ma2016dynamic}]\label{theo:completeGraded}
$\mbf{GrK}$ is sound and complete with respect to the class of all graded frames.
\end{thm}

\subsection{Monotonic modal logic}\label{section:classical}

We consider monotonic modal logic with modalities parametrized by natural numbers, i.e.~$\D_n$ and $\B_n$ with $n\in \mbb{N}$ instead of the usual single modality. As there is no interaction between different $\D_n$ and $\D_m$, the logic for such modalities is not essentially different from the logic for a single modality $\D$ that was originally proposed. 

First, a word on notation. In graded modal logic $\D_n$ denotes the existence of at least $n$ worlds. So in particular $\D$ denotes the existence of at least one world. Whereas in monotonic logic the existence of a neighbourhood is denoted by $\B$ \cite{chellas1980modal} or $\nabla$ \cite{hansen2003monotonic}. We prefer to stick to the notation matching usage in graded modal logic. Therefore also in monotonic modal logic write $\D$ (or $\D_n$) to denote the existence of a neighbourhood instead of $\B$ or $\nabla$ ($\B_n$ or $\nabla_n$). Consequently, the duals of modalities are also swapped.

\medskip

\noindent {\bf Neighbourhood Semantics}.  A \emph{neighbourhood frame} is a tuple $\mathbb{F}= (W,\{\nu_n\}_{n \in \mbb{N}})$ where $W$ is a set of states and each $\nu_n : W\to \mathcal{P}\mathcal{P}(W)$, called \emph{neighbourhood function}.  Denote by $\msf{F}_N$ the class of all neighbourhood frames. A \emph{neighbourhood model} is a pair $\mathbb{M} =(\mathbb{F}, V)$, where $\mathbb{F}$ is a neighbourhood frame and $V:\msf{Prop}\to \mathcal{P}(W)$ is a valuation.

The \emph{truth of a $\mathcal{L}_g$-formula $\varphi$ at a state $w$ of a neighbourhood model $\mathbb{M}=(\mathbb{F},V)$}, notation, $\mathbb{M},w\Vdash \varphi$, is defined inductively as follows, where $n \in \mbb{N}$:
\[
\begin{array}{lcl}
\mathbb{M}, w \Vdash p &\text{~iff~}&  p \in V(p) \\
\mathbb{M}, w\Vdash \neg \psi &\text{~iff~}& \mathbb{M}, w \not \Vdash \psi\\
\mathbb{M}, w \Vdash \psi_1 \lor \psi_2  &\text{~iff~}&\mathbb{M}, w \Vdash \psi_1 \text{~or~} \mathbb{M}, w \Vdash \psi_2 \\
\mathbb{M}, w \Vdash \D_n\psi  &\text{~iff~} & \llbracket \psi \rrbracket_{\mathbb{M}} \in \nu_n(w) 
\end{array}
\]
As an example, Figure \ref{fig:0113} depicts a Kripke model, graded model and a neighbourhood model which all make $\D_3 p$ true.


A neighbourhood function $\nu:W \to \mathcal{P}\mathcal{P}(W)$ is \emph{supplemented} or \emph{closed under supersets} if for all $w \in W$ and $X \subseteq W$, $X\in \nu(w)$ and $X\subseteq Y$ imply $Y \in \nu(w)$. A neighbourhood frame $\mathbb{F}= (W, \{\nu_n\}_{n \in \mbb{N}})$ is \emph{monotonic} if each $\nu_n$ is supplemented. A neighbourhood model $\mathbb{M} = (\mathbb{F},V)$ is \emph{monotonic} if $\mathbb{F}$ is monotonic. Denote by $\msf{F}_M$ the class of all monotonic neighbourhood frames.  Monotonic pointed models $\mathbb{M},w$ and $\mathbb{M}',w'$ are said to be \emph{modally equivalent} 
 if for all $\mathcal{L}_g$-formulas $\varphi$, we have $\mathbb{M},w\Vdash \varphi$ iff $\mathbb{M}',w' \Vdash \varphi$. For monotonic model $\mathbb{M}$, we have
\[
\mathbb{M}, w \Vdash \D_n \varphi \quad \text{~iff~} \quad
\exists X( X\in \nu_n(w) \ \& \ X \subseteq \llbracket \varphi \rrbracket_{\mathbb{M}}).
\]

\medskip

\noindent {\bf Axiomatization}. \ \ The \emph{minimal monotonic modal logic} $\mbf{M_\mbb{N}}$ consists of the following axioms and inference rules, where $n \in \mbb{N}$:
\begin{align*}
&(Ax1)~ \text{all instances of propositional tautologies}\\
&(MP) ~\text{from~} \varphi \text{~and ~}\varphi \to\psi \text{~infer~} \psi\\
&(RM_n) ~\text{from~} \varphi \to \psi \text{~infer~} \D_n \varphi \to \D_n \psi
\end{align*}

 The set of theorems derivable in the system $\mbf{M_\mbb{N}}$ is also called $\mbf{M_\mbb{N}}$. A \emph{monotonic modal logic} is a set $\Lambda$ of $\mathcal{L}_\mbb{N}$-formulas with $\mbf{M_\mbb{N}} \subseteq \Lambda$.
If $\varphi \in \Lambda$, we write $\vdash_\Lambda \varphi$.
\begin{thm}[Theorem 2.41 of \cite{pacuit2017neighborhood}]
$\mbf{M_\mbb{N}}$ is sound and strongly complete with respect to $\msf{F}_M$.  
\end{thm}

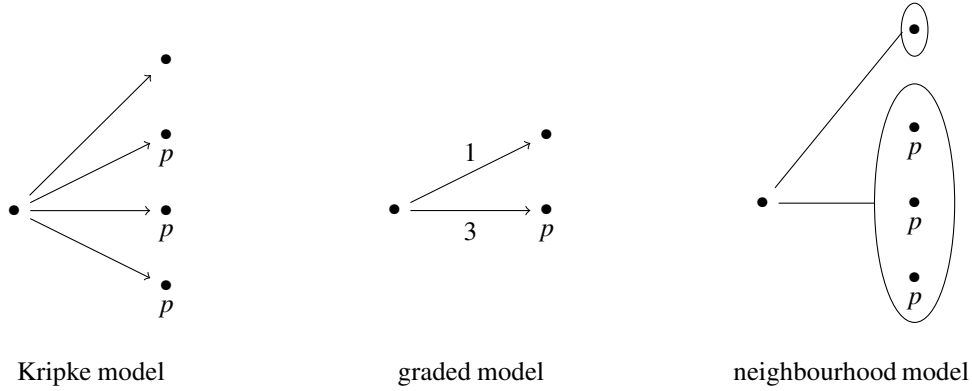
\begin{figure}
\centering
\scalebox{1}{
\begin{tikzpicture}		
\node at (-5,0) {
\begin{tikzpicture}
\node at(0,0) (1) {$\bullet$};
\node at(2,2) (5){$\bullet$};  
\node at(2,1) (2) {$\bullet$};
\node at(2,0) (3){$\bullet$};
\node at(2,-1) (4){$\bullet$};             
\draw[->] (1)-- (2);
\draw [->](1)-- (3);
\draw [->](1)-- (4);
\draw [->](1)-- (5);
\node at(2,0.7) (2) {$p$};
\node at(2,-0.3) (2) {$p$};
\node at(2,-1.3) (2) {$p$};

\end{tikzpicture}
};
\node at (0,0) {
\begin{tikzpicture}
\node at(0,0) (1) {$\bullet$};
\node at(2,1) (2) {$\bullet$};
\node at(2,0) (3){$\bullet$};         
\draw[->] (1)-- node[above=0.3mm]{1}(2);
\draw [->](1)--  node[below=0.3mm]{3}(3);
\node at(2,-0.3)  {$p$};
\end{tikzpicture}
};
\node at (5,0.3) {
\begin{tikzpicture}
\node at(0,0) (1) {$\bullet$};
\node at(2,2.3) (5){$\bullet$};  
\node at(2,1) (2) {$\bullet$};
\node at(2,0) (3){$\bullet$};
\node at(2,-1) (4){$\bullet$};             

\draw (1)-- (1.47,0);
\draw (1)-- (1.84, 2.27);

\node at(2,0.7) {$p$};
\node at(2,-0.3)  {$p$};
\node at(2, -1.3)  {$p$};
\draw (2,0)  ellipse (15pt and 45pt);
\draw (2,2.3)  ellipse (5pt and 10pt);

\end{tikzpicture}
};
\node at(-5,-2.5) (2) {Kripke model};
\node at(0,-2.5) (2) {graded model};
\node at(5,-2.5) (2) {neighbourhood model};
\end{tikzpicture}
}
\caption{Three different ways to make $\D_3 p$ true}
\label{fig:0113}
\end{figure}

\section{Graded modal logics are monotonic modal logics}
\label{sec:3}
In this section we show that graded modal logics are monotonic modal logics. 
Let $\mbf{G}$ be a graded modal logic. 

\begin{prop} \label{prop.gradmon}
Graded modal logics are monotonic modal logics.
\end{prop}

\begin{proof}
Let $\mbf{G}$ be a graded modal logic. To show that $\mbf{G}$ is a monotonic modal logic, it suffices to show that (i) $\mbf{G}$ is closed under $(MP)$ and (ii) for all $n \in \mbb{N}$, $\mbf{G}$ is closed under $(RM_n)$. Item (i) is immediate. We now show item (ii). We distinguish the case $n=0$ from the case $n>0$.

Let $n = 0$. Assume that $\mbf{G} \vdash \varphi \to \psi$. By $(Ax2)$, we have $\D_0 \varphi \leftrightarrow \top$ and $\D_0 \psi \leftrightarrow \top$ and hence $\D_0 \varphi \to \top$ and $\top \to \D_0 \psi$. It follows that $\mbf{G} \vdash  \D_0 \varphi \to  \D_0 \psi$.

Let now $n>0$. Assume that $\mbf{G}\vdash \varphi \to \psi$. By $(Gen)$, $\mbf{G} \vdash \B (\varphi \to \psi)$. Then by $(Ax5)$, $\mbf{G} \vdash \B (\varphi \to \psi)\to ( \D_n\varphi \to \D_n\psi)$. Finally, by $(MP)$ we get $\mbf{G} \vdash  \D_n\varphi \to \D_n\psi$.
\end{proof}

\begin{cor}
$\mbf{GrK}$ is a monotonic modal logic.
\end{cor}

We now define axiomatization $\mbf{GrK}_{Mon}$ as the extension of $\mbf{M}_\mathbb{N}$ with $(Ax2)-(Ax6)$ of $\mbf{GrK}$ and the novel axiom $(Ax7)~ \D(\varphi \vee\psi) \leftrightarrow \D\varphi \vee \D\psi$. We show that $\mbf{GrK}$ and $\mbf{GrK}_{Mon}$ derive the same theorems.

\begin{prop}
For any formula $\varphi$, $\mbf{GrK} \vdash \varphi$ \,iff\, $\mbf{GrK}_{Mon}\vdash \varphi$.
\end{prop}

\begin{proof}
$(\Leftarrow)$\, $(Gen)$ is derivable in $\mbf{GrK}_{Mon}$ as follows:
\[
\begin{array}{lll}
{\fns 1} & \varphi & ~~\text{\fns assumption}\\
{\fns 2} & \varphi \to (\neg \varphi \to \bot) & ~~\text{\fns  Duns Scotus law}\\
{\fns 3} & \neg \varphi \to \bot & ~~\text{\fns 1,2 $(MP)$} \\
{\fns 4} & \D\neg \varphi \to \D \bot & ~~\text{\fns 3 by $(RM_1)$} \\
{\fns 5} & \D \neg \varphi \to \bot & ~~\text{\fns 4 by $(Ax3)$} \\
{\fns 6} & \top \to \neg \D\neg \varphi & ~~\text{\fns 5 by contraposition} \\
{\fns 7} & \Box\varphi & ~~\text{\fns 6 by def.~of $\Box$ and $(Ax1)$} \\
\end{array}
\]
$(\Rightarrow)$\, It suffices to show that $(Ax7)$ is derivable and $(RM_n)$ is admissible rule in $\mbf{GrK}$. The latter follows from Proposition \ref{prop.gradmon}. $(Ax7)$ is equivalent to (i) $\D\varphi \lor \D \psi \to \D (\varphi \lor\psi)$ and (ii) $\D (\varphi \lor\psi) \to \D\varphi \lor \D \psi$. (i) and (ii) are derivable as follows:

\[
\begin{array}{lll}
{\fns 1} & \Box (\varphi \to \varphi \lor \psi) & ~~\text{\fns by $(Ax1)$ and (Gen)} \\
{\fns 2} & \D\varphi \to\D(\varphi \lor \psi) & ~~\text{\fns 1 and (Ax5) by (MP)} \\
{\fns 3} & \Box (\psi \to \varphi \lor \psi) & ~~\text{\fns by $(Ax1)$ and (Gen)} \\
{\fns 4} & \D\psi \to\D(\varphi \lor \psi) & ~~\text{\fns 3 and (Ax5) by (MP)} \\
{\fns 5} & \D\varphi \lor \D \psi \to \D (\varphi \lor \psi)  & ~~\text{\fns 2 and 4 by (Ax1)} \\
\end{array}
\]
\[
\begin{array}{lll}
{\fns 1} & \neg \D (\varphi \land \psi) \land \D_{0} \varphi \land \neg \D \varphi \land \D_{0}\psi \land \neg \D \psi \\ & \ \hspace{2cm} \to \D_{0}(\varphi \lor \psi) \land \neg \D (\varphi \lor \psi)& ~~\text{\fns (Ax6) with $m=n =0$} \\
{\fns 2} & \neg \D (\varphi \land \psi) \land \neg \D \varphi \land \neg \D \psi \to \neg \D (\varphi \lor \psi) & ~~\text{\fns 1 by $(Ax2)$ and $\top \land \varphi \leftrightarrow \varphi$} \\

{\fns 3} & \D (\varphi \lor \psi) \to \D (\varphi \land \psi) \lor \D \varphi \lor \D \psi & ~~\text{\fns 2 by contraposition, De Morgan and double negation} \\

{\fns 4} & \varphi \land \psi \to \varphi  & ~~\text{\fns classical tautology} \\

{\fns 5} & \D(\varphi \land \psi) \to \D \varphi  & ~~\text{\fns 4, $RM_1$} \\

{\fns 6} & \D(\varphi \land \psi) \to \D \varphi \lor \D \psi  & ~~\text{\fns 5, property of $\lor$} \\

{\fns 7} & \D \varphi \to \D \varphi \lor \D \psi  & ~~\text{classical tautology} \\

{\fns 8} & \D \psi \to \D \varphi \lor \D \psi  & ~~\text{classical tautology} \\

{\fns 9} &  \D(\varphi \land \psi) \lor \D \varphi \lor \D \psi \to \D \varphi \lor \D \psi & ~~\text{\fns 6, 7, 8, property of $\lor$ } \\

{\fns 10} & \D (\varphi \lor\psi) \to \D \varphi \lor \D \psi & ~~\text{\fns 3, 9, hypothetical syllogism} \\
\end{array}
\]

\end{proof}
Another interesting question is whether there exists a class of neighbourhood frames with respect to which $\mbf{GrK}$ is sound and complete. In monotonic neighbourhood frames the class of so-called \emph{KW-formulas} (\cite[Definition 5.13]{hansen2003monotonic}) is elementary (\cite[Theorem 5.14]{hansen2003monotonic} and canonical (\cite[Theorem 10.34]{hansen2003monotonic}). Therefore, a presentation where each axiom is a KW-formula would make it straightforward to prove soundness and strong completeness. 
Unfortunately, (Ax5) and (Ax6) are not KW-formulas, since they have $\neg $ inside the scope of $\D$, which is forbidden in KW-formulas. Therefore we can not prove completeness of $\mbf{GrK}$ indirectly via a reference to KW-formulas. 

If we adopt a more direct method to prove the completeness, we need to show that the properties defined by (Ax2)-(Ax7) holds in the canonical frame of monotonic modal logic containing them. Axioms (Ax5) and (Ax6) resp.\ correspond to the  properties:
\[\begin{array}{l}
\hspace{-.2cm}\forall w  \forall X \forall Y(X\cap (W \! \setminus \! Y) \not\in \nu_1(w)\ \& \ X \in \nu_n(w) \Rightarrow Y \in \nu_n(w)) \\
\hspace{-.2cm}\forall w  \forall X \forall Y (X \!\cap\! Y \not \in \nu_1(w) \ \& \ X \in \nu_m(w) \ \& \ X \not \in \nu_{m+1}(w) \ \& \  Y \in \nu_n(w) \ \& \ Y \not \in \nu_{n+1}(w) \Rightarrow \\ \ \hspace{7cm}
 X\! \cup \!Y \in \nu_{m+n}(w) \ \& \ X \cup Y \not \in \nu_{(m+n+1)}(w))
\end{array}\]
The difficulty lies at showing that (Ax5) and (Ax6) are valid in the canonical frame of monotonic modal logic containing (Ax5) and (Ax6). For canonical frames of monotonic modal logics, we refer to \cite[Def. 9.3]{chellas1980modal},  \cite[Def. 6.2]{hansen2003monotonic} and \cite[Def. 2.37]{pacuit2017neighborhood}.

In the next section, we identify a class of complete neighbourhood frames via an operation $(.)^\bullet$, which is shown to be first-order definable in Section \ref{sec:FOdef} and modally undefinable in Section \ref{sec:modalundef}.

\section{Graded neighbourhood frames}\label{sec:gnf}

Given a set $X$, denote by $\mathcal{P}_{\ge n} (X)$ the set of subsets of $X$ such that the cardinality of each subset is at least $n$, in other words,
 $\mathcal{P}_{\ge n} (X)=\{ X' \subseteq X \mid \abs{X'} \ge n \}$. For $\Gamma \subseteq \mathcal{P}(W)$, define $\uparrow \! \! \! \Gamma$ to be the up-set generated by $\Gamma$, that is, $\uparrow  \!  \!\!  \Gamma:= \{Y\in \mathcal{P}(W) \mid  \exists X(X \in \Gamma \ \& \ X\subseteq Y)\}$.
 \begin{defn}\label{def:1744}
A neighbourhood frame $\mathbb{F}=(W,\{\nu_n\}_{n \in \mathbb{N}})$ is a \emph{graded neighbourhood frame} if for all $w \in W$, there exists an $A \subseteq W$ such that for all $ n\in \mathbb{N}$, 
 ${\nu_n(w)} = {\uparrow \! \!  \mathcal{P}_{\ge n}(A)}$.\end{defn}

\begin{defn}\label{defn:bullet}
For a Kripke frame $\mathcal{F}=(W,R)$, the associated  graded neighbourhood frame of $\mathcal{F}$ is $\mathcal{F}^\bullet=(W, \{\nu_n\}_{n\in \mathbb{N}})$, where for $w \in W$ and $n \in \mathbb{N}$, 
${\nu_n(w)} = {\uparrow \! \! \mathcal{P}_{\ge n} (R[w])}$.
\end{defn}
That each $\nu_n$ in $\mathcal{F}^\bullet=(W, \{\nu_n\}_{n\in \mathbb{N}})$ is monotonic follows directly from the definition. Then we have the following result:

\begin{prop}\label{prop:0004}
Let $\mathcal{F}=(W,R)$ be a Kripke frame and $V$ a valuation on $\mathcal{F}$. Then for all $w\in W$ and all formulas $\varphi$
\[
(\mathcal{F},V),w \Vdash \varphi \quad \text{~iff~}\quad (\mathcal{F}^\bullet,V) ,w \Vdash \varphi. 
\]
\end{prop}

\begin{proof}
The proof is by induction on $\varphi$. The propositional cases follows from the definition and induction hypothesis.

As for the modal case, let $\varphi$ be $\D_n \psi, n \in \mbb{N}$, we have
\[
\begin{array}{lcll}
(\mathcal{F},V),w \Vdash \D_n \psi&
\text{~iff~} & \abs{R[w]\cap \den{\psi}_{(\mathcal{F},V)}}\ge n\\
&\text{~iff~} & \abs{R[w]\cap \den {\psi}_{(\mathcal{F}^\bullet,V)}}\ge n &\text{(IH)}\\
&\text{~iff~} & \exists X \subseteq W \ (X \in \nu_n(w)  \ \& \ X \subseteq  \den {\psi}_{(\mathcal{F}^\bullet,V)})&(*)\\
&\text{~iff~} & (\mathcal{F}^\bullet,V) , w \Vdash \D_n \psi
\end{array}
\]
Here is the proof for the equivalence marked by $(*)$. First assume that  $\abs{R[w]\cap \den {\psi}_{(\mathcal{F}^\bullet,V)}}$ $ \ge n$. Then $R[w]\cap\den {\psi}_{(\mathcal{F}^\bullet,V)}\in \mathcal{P}_{\ge n}(R[w])$. By definition, $\nu_n(w) = \uparrow \! \! \mathcal{P}_{\ge n}(R[w])$. Hence, $R[w]\cap \den {\psi}_{(\mathcal{F}^\bullet,V)}\in \nu_n(w)$. We also have $R[w]\cap \den {\psi}_{(\mathcal{F}^\bullet,V)}\subseteq \den {\psi}_{(\mathcal{F}^\bullet,V)}$, which completes the proof of this direction. Now assume that $X \in \nu_n(w)$ and $X \subseteq \den {\psi}_{(\mathcal{F}^\bullet,V)}$. Since $\nu_n(w)=\uparrow \! \! \mathcal{P}_{\ge n} (R[w])$, $X \in \uparrow \! \! \mathcal{P}_{\ge n}(R[w])$. Then there exists $Y \in \mathcal{P}_{\ge n}(R[w])$ and $Y\subseteq X$. It follows that $Y \subseteq R[w]$ and $\abs{Y} \ge n$. Since $X \subseteq\den {\psi}_{(\mathcal{F}^\bullet,V)}$, $Y \subseteq\den {\psi}_{(\mathcal{F}^\bullet,V)}$. Hence, $Y = Y\cap \den {\psi}_{(\mathcal{F}^\bullet,V)} \subseteq R[w] \cap \den {\psi}_{(\mathcal{F}^\bullet,V)}$ and therefore $\abs{R[w] \cap \den {\psi}_{(\mathcal{F}^\bullet,V)}}\ge \abs{Y} \geq n$.
\end{proof}

Given a graded neighbourhood frame $\mathbb{F}= (W,\{\nu_n\}_{n \in \mathbb{N}})$ with  $\nu_n(w)  = \uparrow \!\! \mathcal{P}_{\ge n}(A_w)$, we can associate it with a Kripke frame $\mathbb{F}_\bullet =(W,R)$ with $R[w]=A_w$. It follows from definitions that $(\mathbb{F}_\bullet)^\bullet = \mathbb{F}$ and $(\mathcal{F}^\bullet)_\bullet =\mathcal{F}$.

For a class of Kripke frames $S_K$, let $S_K^\bullet=\{\mathcal{F}^\bullet \mid \mathcal{F} \in S_K\}$. Recall that $\msf{F}_K$ is the class of all Kripke frames. Since $(\mathbb{F}_\bullet)^\bullet = \mathbb{F}$ for any graded neighbourhood frame $\mathbb{F}$, $\msf{F}_K^\bullet$ is equivalent to the class of all graded neighbourhood frames. 

\begin{thm}\label{thm:0009}
$\mbf{GrK}$ is sound and strongly complete with respect to the class of graded neighbourhood frames.
\end{thm}

\begin{proof}
By Theorem \ref{thm:2106}, $\mbf{GrK}$ is sound and strongly complete with respect to $\msf{F}_K$. By Proposition \ref{prop:0004}, $\mbf{GrK}$ is sound and strongly complete with respect to $\msf{F}_K^\bullet$. Then the claim follows from the fact that $\msf{F}_K^\bullet$ is equivalent to the class of all graded neighbourhood frames. 
\end{proof}

\section{Graded neighbourhood frames are first-order definable}\label{sec:FOdef}

 A class $S_N$ of neighbourhood frames is \emph{first-order definable} if there exists a set of first-order formulas $\Gamma$ such that $\mathbb{F} \models \Gamma$ iff $\mathbb{F} \in S_N$. In this section, we show that the class of graded neighbourhood frames is(two-sorted) first-order definable in the (two-sorted) first-order language $\mathcal{L}_g^1$ of $\mathcal{L}_g$ defined below.

Each monotonic neighbourhood frame $\mathbb{F}=(W, \{\nu_n\}_{n \in \mathbb{N}})$ can be seen as a two-sorted relational structure $(W,\mathcal{P}(W), \{R_{\nu_n}\}_{n \in \mathbb{N}}, R_\ni)$ where $R_{\nu_n} \subseteq W \times \mathcal{P}(W)$ and $R_\ni \subseteq \mathcal{P}(W) \times W$ such that $w R_{\nu_n} X$ iff $X\in \nu_n(w)$ and $XR_\ni w$ iff $w\in X$. Accordingly, the (two-sorted) first-order language $\mathcal{L}^1_g$ of $\mathcal{L}_g$ has equality $=$, first-order variables $w, u,v,\ldots$ over $W$,  first-order variables $X,Y,Z,\ldots$ over $\mathcal{P}(W)$, binary symbols $R_{\nu_n}$ for $n\in \mathbb{N}$ and $R_\ni$, and unary relation symbols $P,Q,\ldots$ corresponding to $p,q,\ldots\in \msf{Prop}$.

 In other words, given sets of variables $\Psi$ and $\Phi$, formulas in $\mathcal{L}_g^1$ are defined inductively as follows:
 \[
 \mathcal{L}_g^1 \ni \chi ::  = w = u  \mid X= Y \mid Pw \mid R_{\nu_n} wX \mid R_\ni Xw \mid \neg \chi \mid \chi \lor \chi \mid \forall x \chi  \mid \forall X \chi
 \]
 where $w, u \in \Psi$, $X, Y \in \Phi$, $P$ corresponds to $p \in \msf{Prop}$ and $n \in \mathbb{N}$.

A set $A$ is called \emph{atomic in $\nu_1(w)$} if for all $a \in A$, $\{a\} \in \nu_1(w)$.
Denote by $(\star)$ the following conditions: for all $w \in W$
\begin{itemize}
\item [$(\star1)$] $\nu_0(w) =\mathcal{P}(W)$.
\item [$(\star2)$] $\nu_n(w)$ is closed under supersets for $n \in \mathbb{N}$.
\item [$(\star3)$] $\emptyset \not \in \nu_n(w)$ for $n \in \mathbb{N}$.
\item [$(\star4)$] If $X\in \nu_n( w )$, then there exists a minimal $Y \in \nu_n( w )$ such that $Y\subseteq X$.
\item [$(\star5)$] If $Y$ is a minimal element in $\nu_n(w)$, then $\abs{Y}= n$ and  $Y$ is atomic in $\nu_1(w)$. 
\item [$(\star6)$] If $\{y_1\},\ldots, \{y_n \}\in \nu_1(w)$ and $y_1 ,\ldots,  y_n$ are pairwise distinct, then $\bigcup_{1 \le i \le n} \{y_i\}$ is a minimal element in $\nu_n(w)$.
\end{itemize}
Note that conditions $(\star)$ can be expressed in language $\mathcal{L}_g^1$. For example, $\abs{Y}\ge n$ iff $y_1\in Y\land\ldots\land y_n\in Y\land  \bigwedge_{i\neq j}y_i\neq y_j$, and $Y$ is atomic in $\nu_1(w)$ iff
$\forall Z( \forall Z'(Z' \!\subseteq \!Z \Rightarrow Z'\! =\!\emptyset \text{~or~} Z'\! = \!Z) \ \& \ Z \! \subseteq \!Y \Rightarrow Z \in \nu_1(w))$.

\begin{prop}
Let $\mathbb{F}=(W, \{\nu_n\}_{n\in \mathbb{N}})$ be a neighbourhood frame. Then $\mathbb{F}$ is graded \weg{a graded neighbourhood frame} iff $\mathbb{F}$ satisfies $(\star)$.
\end{prop} 

\begin{proof}
For the left-to-right direction, assume that $\mathbb{F}=(W,\{\nu_n\}_{n \in \mathbb{N}})$ is a graded neighbourhood frame, that is, for all $w \in W$, there exists some $A \subseteq W$ such that for all $n \in \mathbb{N}$, $\nu_n(w)= \uparrow \!\! \mathcal{P}_{\ge n}(A)$. Since $\uparrow \!\! \mathcal{P}_{\ge 0}(A) = \uparrow \!\!\mathcal{P}(A)= \mathcal{P}(W) $, item $(\star1)$ holds. Item ($\star2$) and ($\star3$) also follow directly. 

Now assume that $X\in \nu_n( w )$. Since $\nu_n(w)  = \uparrow \!\! \mathcal{P}_{\ge n} (A)$, there exists $Y \in \mathcal{P}_{\ge n} (A)$ with $Y \subseteq X$. It follows that $\abs{Y}\ge n$. Let $Y'$ be a subset of $Y$ containing exactly $n$-elements. Then $Y'$ is a minimal element in $\nu_n(w)$ and $Y' \subseteq X$. Hence, item ($\star4$) follows.

 Now assume that $Y$ is a minimal element in $\nu_n(w) = \uparrow \!\! \mathcal{P}_{\ge n}(A)$. Then $Y\subseteq A$ and $\abs{Y}=n$. Since $\nu_1(w)=\uparrow \!\! \mathcal{P}_{\ge 1}(A)$, for all $a \in A$, $\{a\} \in \nu_1(w)$. It follows that $Y$ is atomic in $\nu_1(w)$. Hence, item $(\star5)$ holds. For item ($\star6)$, assume that $\{y_1\}\not= \ldots \not =\{y_n\} \in \nu_1(w)=\uparrow \!\! \mathcal{P}_{\ge 1}(A)$. Then $\{y_1,\ldots ,y_n\} \in \uparrow \!\!\mathcal{P}_{\ge n}(A)$. It follows that $\{y_1,\ldots ,y_n\}$ is a minimal element in $\nu_n(w)$. Hence, item ($\star6$) holds.

The right-to-left direction follows from Lemma \ref{lem:0042} and \ref{lem:0043} below.
\end{proof}

\begin{lem}\label{lem:0005}
Let $\mathbb{F}=(W, \{\nu_n\}_{n\in \mathbb{N}})$ be a neighbourhood frame satisfying $(\star)$. If $X \in \nu_1(w)$, there exists $x \in X$ such that $\{x\} \in \nu_1(w)$.
\end{lem}

\begin{proof}
Assume that $X\in \nu_1(w)$. By $(\star 4)$, there exists a minimal $Y \in \nu_1(w)$ such that $Y\subseteq X$.  By $(\star 3)$, $X\not = \emptyset$ and $Y\not = \emptyset$. By $(\star 5)$, $Y$ is atomic in $\nu_1(w)$, i.e., for all $y \in Y$, $\{y \} \in \nu_1(w)$. It follows that there exists $x \in X$ such that $\{x\} \in \nu_1(w)$.
\end{proof}

\begin{lem}
Let $\mathbb{F}=(W, \{\nu_n\}_{n\in \mathbb{N}})$ be a neighbourhood frame satisfying $(\star)$. If $\nu_1(w)\neq\emptyset$, there exists a set $A\subseteq W$ such that $A$ is the maximum atomic set in $\nu_1(w)$.
\end{lem}

\begin{proof}
Since $\nu_1(w)\neq\emptyset$, we assume $X \in \nu_1(w)$. By $(\star 3)$, $X \neq \emptyset$. By $(\star 4)$, there exists a minimal $X' \in \nu_1(w)$ such that $X' \subseteq A$. By $(\star 5)$, $\abs{X'} = 1$ and $X'$ is atomic in $\nu_1(w)$. Hence, we can assume $X'= \{a\}$. Let $A$ be the union of all singletons in $\nu_1(w)$. Since $\{a\} \in \nu_1(w)$,  $A \not = \emptyset$. Now we show that $A$ is the maximum atomic set in $\nu_1(w)$. Since $A$ is the union of all singletons in $\nu_1(w)$, $A$ is atomic. Let $B$ be an atomic set in $\nu_1(w)$. For any $b \in B$, by atomicity, $\{b\} \in \nu_1(w)$. It follows that $b \in A$. Therefore, $B \subseteq A$. Hence, $A$ is the maximum atomic set in $\nu_1(w)$.
\end{proof}

\begin{lem}\label{lem:0042}
Let $\mathbb{F}=(W, \{\nu_n\}_{n\in \mathbb{N}})$ be a neighbourhood frame satisfying $(\star)$. If $\nu_1(w)\neq\emptyset$, then $\nu_1(w) = \uparrow \! \! \mathcal{P}_{\ge 1}(A) $, where $A$ is the maximum atomic set in $\nu_1(w)$.
\end{lem}

\begin{proof}
If $\nu_1(w) =\emptyset$,  then $A = \emptyset$. Then $\nu_1(w) = \uparrow \! \! \mathcal{P}_{\ge 1}(A)$. If $\nu_1(w) \not = \emptyset$,  assume that $X \in \nu_1(w)$. By Lemma \ref{lem:0005}, there exists an $x \in X$ such that $\{x\}\in \nu_1(w)$. Since $A$ is the maximum atomic set in $A$, we have $x \in A$. It follows that $\{x\} \in \mathcal{P}_{\ge 1}(A)$. Since $x \in X$, ${X} \in {\uparrow \! \! \mathcal{P}_{\ge 1}(A)}$.

Assume that ${X} \in  {\uparrow \! \! \mathcal{P}_{\ge 1}(A)}$. Then there exists $Y\in \mathcal{P}_1(A)$ such that $Y\subseteq X$. Since $A$ is atomic in $\nu_1(w)$, for all $y\in Y$, $\{y\} \in \nu_1(w)$. By ($\star2$), $\nu_1(w)$ is monotonic. Therefore, $Y\in \nu_1(w)$. Since $Y \subseteq X$, $X \in \nu_1(w)$.
\end{proof}

\begin{lem}\label{lem:0043}
Let $\mathbb{F}=(W, \{\nu_n\}_{n\in \mathbb{N}})$ be a neighbourhood frame satisfying $(\star)$.  Then for $w \in W$,
\begin{itemize}
\item[1.] If $\nu_1(w) =\emptyset$, then $\nu_n(w) = \emptyset$ for $n >1$.
\item[2.] If $\nu_1(w) \not =\emptyset$, then 
 $\nu_n(w) = \uparrow \! \! \mathcal{P}_{\ge n}(A)$ for $n>1$, where $A$ is the maximum atomic set in $\nu_1(w)$.
 \end{itemize}
\end{lem}

\begin{proof}

For item 1, we prove by contradiction. Assume that $\nu_1(w) = \emptyset$ and for some $n >1$, $X \in \nu_n(w)$. By $(\star 3)$, $X \not = \emptyset$. By $(\star 4)$ and $(\star 5)$, there exists $X' \subseteq X$ such that $X'$ is atomic in $\nu_1(w)$.  By $(\star 3)$, $X' \not = \emptyset$. By atomicity of $X'$, $\nu_1(w) \not = \emptyset$, contradiction .

Now we prove item 2 and assume that $X \in \nu_n(w)$. By $(\star 4)$, there exists a minimal element of $\nu_n(w)$ such that $Y\subseteq X$. By $(\star5)$, $\abs{Y} \ge n$ and $Y$ is atomic in $\nu_1(w)$. Since $A$ is the maximum atomic set of $\nu_1(w)$, $Y\subseteq A$. Since $\abs{Y} \ge n$, $Y \in  \mathcal{P}_{\ge n}(A)$. Since $Y \subseteq X$, $X \in \uparrow \! \! \mathcal{P}_{\ge n}(A)$.

Assume that $X \in \uparrow \! \! \mathcal{P}_{\ge n} (A)$. Then there exists $Y \in  \mathcal{P}_{\ge n} (A)$ such that $Y \subseteq X$. It follows that $\abs{Y} \ge n$. Since $A$ is the maximum atomic set of $\nu_1(w)$, $Y$ is atomic in $\nu_1(w)$. Hence, there exist distinct $y_1,\ldots ,y_n \in Y$ such that $\{y_1\}, \ldots, \{y_n\}\in \nu_1(w)$ and $y_1 \not = \ldots \not = y_n$. By $(\star 6)$, $\bigcup_{1\le i \le n}\{y_i\}$ is a minimal element in $\nu_n(w)$. Since $\bigcup_{1\le i \le n}\{y_i\}\subseteq Y \subseteq X$ and $\nu_n(w)$ is monotonic by ($\star2$), $X \in \nu(w)$.
\end{proof}

\section{Graded neighbourhood frames are not modally definable}\label{sec:modalundef}

 A class $S_N$ of neighbourhood frames is \emph{modally definable} if there exists a set of modal formulas $\Delta$ such that $\mathbb{F} \Vdash \Delta$ iff $\mathbb{F} \in S_N$. In this section, we show that the class of graded neighbourhood frames is not modally definable. It is well known that if the class of neighbourhood frames is modally definable, then it is closed under bounded morphic images. Below we show that the class of graded neighbourhood frames is not closed under bounded morphic images (by exhibiting a counterexample), so we conclude that it is not modally definable.

Given a function $f:W\to W'$ and $X\subseteq W$, define $f[X]:= \{f(x) : x\in X  \}$.
\begin{defn}\label{defn:1801}
Let $\mathbb{F}=(W,\{\nu_n \}_{n\in \mathbb{N}})$ and $\mathbb{F}'=(W,\{\nu'_n \}_{n\in \mathbb{N}})$ be neighbourhood frames.  
A \emph{bounded morphism} from $\mathbb{F}$ to $\mathbb{F}'$  is a function $f:W \to W'$ satisfying for $n \in \mathbb{N}$

$(BM1_n)$ If $X \in \nu_n (w)$, then $f[X] \in \nu'_n (f(w))$. 

$(BM2_n)$ If $X' \in \nu'_n(f(w))$, then there exists $X \subseteq W$ such that $f[X] \subseteq X'$ and $X\in \nu(w)$.

If there is a surjective bounded morphism from $\mathbb{F}$ to $\mathbb{F'}$, we say that $\mathbb{F'}$ is a \emph{bounded morphic image} of $\mathbb{F}$.
\end{defn}

\begin{prop}[Prop. 5.3 of \cite{hansen2003monotonic}]\label{prop:0013}
Let $\mathbb{F}$ and $\mathbb{F}'$ be neighbourhood frames. If $\mathbb{F'}$ is a bounded morphic image of $\mathbb{F}$, then $\mathbb{F} \Vdash \varphi$ implies $\mathbb{F}' \Vdash \varphi$.
\end{prop}

\begin{prop}\label{prop:0030}
If a class of neighbourhood frames is modally definable, then it is closed under bounded morphic images.
\end{prop}

\begin{proof}
Let $S_N$ be a class of neighbourhood frames defined by a set of formulas $\Delta$, $\mathbb{F}\in S_N$ and $\mathbb{F}'$ a bounded morphic image of $\mathbb{F}$.
Since $\mathbb{F}\in S_N$, $\mathbb{F}\Vdash \Delta$. By Proposition \ref{prop:0013}, $\mathbb{F}'\Vdash \Delta$ and therefore $\mathbb{F}' \in S_N$.
\end{proof}

\begin{exam}\label{exam:0034}
Consider neighbourhood frames $\mathbb{F}=(\{a,b\},\{\nu_n\}_{n \in \mathbb{N}})$ such that for $n \in \mathbb{N}$, $\nu_n(a)=\nu_n(b)= {\uparrow \! \! \mathcal{P}_{\ge n}(\{a,b\})} $ and  $\mathbb{F}'=(\{c\}, \{\nu'_n\}_{n \in \mathbb{N}})$ such that $\nu'_0(c)= \{ \emptyset, \{c\}\}$, $\nu'_1(c)=\nu'_2(c)= \{ \{c\}\}$ and $\nu'_k(c)=\emptyset$ for $k > 2$. By Definition \ref{def:1744}, $\mathbb{F}$ is a graded neighbourhood frame. As for $\mathbb{F}'$, we have ${\nu_1(c)} = {\uparrow \!\! \mathcal{P}_{\ge1} (\{c\})}$ while ${\nu_2(c)} \not =  {\uparrow \!\! \mathcal{P}_{\ge2} (\{c\})}$. Therefore, $\mathbb{F}'$ is not a graded neighbourhood frame. It can be verified that function $f: \{a ,b \} \to \{ c\}$, with $f(a) =f(b) =c$, is a subjective bounded morphism from $\mathbb{F}$ to $\mathbb{F}'$. Therefore, the class of graded neighbourhood frames is not closed under bounded morphic images. 
\end{exam}



\begin{prop}
The class of graded neighbourhood frames is not modally definable.
\end{prop}

\begin{proof}
It follows from Example \ref{exam:0034} and the contraposition of Proposition \ref{prop:0030}. 
\end{proof}

\section{Bisimulation}\label{sec:bisimulation}
The notion of graded tuple bisimulation was first proposed in de Rijke \cite{de2000note}. In this section, we obtain a new definition of graded bisimulation by substituting $\nu_n(w)$ with $\uparrow \! \! \mathcal{P}_{\ge n}(R[w])$ in the definition of monotonic bisimulation. And we prove that the new definition is equivalent to the old one (cf. Proposition \ref{prop:1812} and \ref{prop:1813}). 

\subsection{From monotonic bisimulation to graded bisimulation}
\begin{defn}
[Monotonic bisimulation, Def. 4.10 of \cite{hansen2003monotonic}] \label{defn:mon}
Suppose that $\mathbb{M}=(W,\{\nu_n\}_{n \in \mathbb{N}},V)$ and $\mathbb{M}'=(W',\{\nu'_n\}_{n \in \mathbb{N}},V')$ are monotonic neighbourhood models. A non-empty relation $Z\subseteq W\times W'$ is a \emph{monotonic bisimulation} (notation: $Z:\mathbb{M} \leftrightarroweq_m \mathbb{M}'$) provided that
\begin{itemize}
\item{(\bf{Prop})} If $wZw'$, then $w$ and $w'$ satisfy the same proposition letters.

\item{(\bf{Forth})} If $wZw'$ and $X \in \nu_n(w)$, then there is $X'\subseteq W'$ such that $X' \in \nu_n'(w')$ and $\forall x' \in X' \exists x\in X: xZx'$.

\item{(\bf{Back})} If $wZw'$ and $X' \in \nu_n'(w')$, then there is  $X\subseteq W$ such that $X \in \nu_n(w)$ and $\forall x\in X \exists x'\in X': xZx'$.
\end{itemize}
If $w\in \mathbb{M}$ and $w'\in \mathbb{M'}$, then $w$ and $w'$ are \emph{monotonic bisimilar states} (notation: $\mathbb{M},w \leftrightarroweq_m \mathbb{M'},w'$) if there is a bisimulation $Z:\mathbb{M} \leftrightarroweq_m \mathbb{M}'$ with $wZw'$.
\end{defn}

\begin{prop}[Prop. 4.11 of \cite{hansen2003monotonic}]\label{prop:1932}
Let $\mathbb{M}=(W,\{\nu_n\}_{n \in \mathbb{N}},V)$ and $\mathbb{M}'=(W',\{\nu'_n\}_{n \in \mathbb{N}},V')$ be monotonic neighbourhood models. If $\mathbb{M},w \leftrightarroweq_m \mathbb{M}',w'$, then for $\mathcal{L}_g$-formula $\varphi$, $\mathbb{M},w\Vdash \varphi$ iff $\mathbb{M}',w' \Vdash \varphi$.
\end{prop}

Substituting $\nu_n(w)$ in Definition \ref{defn:mon} with $\uparrow \! \! \mathcal{P}_{\ge n}(R[w])$, we have:
\begin{defn}
[Graded bisimulation]\label{defn:gradedbisi}
 Suppose that $\mathcal{F}=(W,R,V)$ and $\mathcal{M}'=(W',R',V)$ are Kripke models. A  non-empty relation $Z\subseteq W\times W'$ is a \emph{graded bisimulation} (notation: $Z:\mathcal{M} \leftrightarroweq_g \mathcal{M}'$) provided that
\begin{itemize}
\item{(\bf{Prop})} If $wZw'$, then $w$ and $w'$ satisfy the same proposition letters.
\item{(\bf{Forth})} If $wZw'$ and $X \in \uparrow \! \! \mathcal{P}_{\ge n}(R[w])$, then there is an $X'\subseteq W'$ such that $X' \in \uparrow \! \! \mathcal{P}_{\ge n}(R'[w'])$ and $\forall x' \in X' \exists x\in X: xZx'$.

\item{(\bf{Back})} If $wZw'$ and $X' \in \uparrow \! \! \mathcal{P}_{\ge n}(R'[w'])$, then there is an $X\subseteq W$ such that $X \in \uparrow \! \! \mathcal{P}_{\ge n}(R[w])$ and $\forall x\in X \exists x'\in X': xZx'$.
\end{itemize}
If $w\in \mathcal{M}$ and $w'\in \mathcal{M'}$, then $w$ and $w'$ are \emph{graded bisimilar states} (notation: $\mathcal{M},w \leftrightarroweq_g \mathcal{M'},w'$) if there is a bisimulation $Z:\mathcal{M} \leftrightarroweq_g \mathcal{M}'$ with $wZw'$.
\end{defn}

\begin{prop}\label{prop:1958}
Let $\mathcal{M}=(W,R,V)$ and $\mathcal{M}'=(W',R',V')$ be Kripke models. If $\mathcal{M},u \leftrightarroweq_g \mathcal{M}',u'$, then   $\mathcal{M},u\equiv_k \mathcal{M}',u'$.
\end{prop}

\begin{proof}
Since $\mathcal{M},u \leftrightarroweq_g \mathcal{M}',u'$, there exists a non-empty relation $Z\subseteq W \times W'$ such that $Z:\mathcal{M} \leftrightarroweq_g \mathcal{M}'$ and $uZu'$. 
For neighbourhood frames $\mathcal{M}^\bullet=(W,\{\nu_n\}_{n \in \mathbb{N}},V)$ and $\mathcal{M'}^\bullet=(W
,\{\nu'_n\}_{n \in \mathbb{N}},V')$, by definition, for $w \in W$ and $w' \in W'$, $\nu_n(w)=  \uparrow \! \! \mathcal{P}_{\ge n}(R[w])$ and $\nu'_n (w') = {\uparrow \! \! \mathcal{P}_{\ge n}(R'[w'])}$. Substituting $ \uparrow \! \! \mathcal{P}_{\ge n}(R[w])$ with $\nu_n(w)$ and $ \uparrow \! \! \mathcal{P}_{\ge n}(R'[w'])$ with $\nu'_n(w')$ in the definition of $Z:\mathcal{M} \leftrightarroweq_g \mathcal{M}'$, we have  $Z:\mathcal{M}^\bullet,u \leftrightarroweq_m \mathcal{M}'^\bullet,u'$ and $uZu'$. For all formulas $\varphi$, that $\mathcal{M},u\Vdash \varphi$ iff $\mathcal{M}',u' \Vdash \varphi$ can be proved as follows:
\[
\begin{array}{lcll}
\mathcal{M},u\Vdash \varphi 
& \text{~~iff~~} & \mathcal{M}^\bullet,u \Vdash \varphi & \text{Proposition}~ \ref{prop:0004}\\
& \text{~~iff~~} & \mathcal{M}'^\bullet,u' \Vdash \varphi & \text{Proposition}~ \ref{prop:1932}\\
& \text{~~iff~~} & \mathcal{M}',u'\Vdash \varphi & \text{Proposition}~ \ref{prop:0004}\\
\end{array}
\]
\end{proof}

\subsection{Graded bisimulation is equivalent to graded tuple bisimulation}
In the rest of this section, we recall the definition of graded tuple bisimulation in de Rijke \cite{de2000note} and show that it is equivalent to Definition \ref{defn:gradedbisi}. Given a set $X$, denote by $\mathcal{P}_{<\omega} (X)$ the set of finite subsets of $X$. We now get:
\begin{defn}[Graded tuple bisimulation]\label{defn:gbt}
Let $\mathcal{M}=(W,R,V)$ and $\mathcal{M}=(W',R',V')$ be two Kripke models. A tuple $\mathcal{Z}=(\mathcal{Z}_1,\mathcal{Z}_2,\ldots)$ of relations is called \emph{graded tuple bisimulation} between $\mathcal{M}$ and $\mathcal{M}'$ (notation: $\mathcal{Z}:\mathcal{M}\leftrightarroweq_{gt} \mathcal{M}')$ iff: 
\begin{itemize}
\item[(1)] $\mathcal{Z}_1$ is non-empty;
\item[(2)] for all $i$, $\mathcal{Z}_i\subseteq \mathcal{P}_{<\omega} (W_1) \times \mathcal{P}_{< \omega}(W_2)$;
\item[(3)] if $X\mathcal{Z}_i X'$, then $\abs{X}=\abs{X'} =i$; 
\item[(4)] if $\{w\}\mathcal{Z}_1 \{w'\}$, then $w$ and $w'$ satisfy the same proposition letters;
\item[(5)] if $\{w\} \mathcal{Z}_1\{w'\}$, $X\subseteq R[w]$ and $\abs{X}=i \ge 1$, then there exists $X' \in \mathcal{P}_{< \omega}(W')$ with $X' \subseteq R'[w']$ and $X\mathcal{Z}_iX'$;
\item[(6)] if $\{w\}\mathcal{Z}_1\{w'\}$, $X'\subseteq R[w']$ and $\abs{X'} = i \ge 1$, then there exists $X\in \mathcal{P}_{<\omega} (W)$ with $X\subseteq R[w]$ and $X\mathcal{Z}_iX'$;
\item[(7)] if $X\mathcal{Z}_iX'$, then (a)  $\forall x\in X \exists x'\in X': \{x\}\mathcal{Z}_1\{x'\}$, and  (b) $\forall x' \in X' \exists x\in X: \{x\}\mathcal{Z}_1\{x'\}$.
\end{itemize}
\end{defn}

\begin{prop}\label{prop:1812}
Let $\mathcal{M}=(W,R,V)$ and $\mathcal{M}'=(W',R',V')$ be Kripke models and $\mathcal{Z}=(\mathcal{Z}_1,\mathcal{Z}_2,\ldots)$ a tuple of relations such that $\mathcal{Z}: \mathcal{M}\leftrightarroweq_{gt}\mathcal{M}'$. Define $Z\subseteq W\times W'$ to be a relation such that $wZw'$ iff $\{w\}\mathcal{Z}_1\{w'\}$. Then $Z:\mathcal{M}\leftrightarroweq_g \mathcal{M}'$.
\end{prop}

\begin{proof}

({\bf Prop}) follows from item (4) of Definition \ref{defn:gbt}. As for ({\bf Forth}), assume that $wZw'$ and $X \in \uparrow \! \! \mathcal{P}_{\ge n}(R[w])$. Then there exists $Y \subseteq R[w]$ such that $Y\subseteq X$ and $\abs{Y}=n$. Since $\abs{Y} =n$ and $\{w\}\mathcal{Z}_1\{w'\}$, by items (5) and (3) there exists $Y' \subseteq R'[w']$, $\abs{Y'}=n$  and $Y\mathcal{Z}_n Y'$. It follows that $Y' \in \uparrow \! \! \mathcal{P}_{\ge n}(R'[w'])$. By item (7)(b), $\forall y' \in Y' \exists y\in Y: \{y\}\mathcal{Z}_1\{y'\}$. Since $Y\subseteq X$ and $xZy$ iff $\{x\}\mathcal{Z}_1\{y\}$, we have $\forall y' \in Y' \exists x\in X: xZy'$, which completes the proof of that $Z$ satisfies ({\bf Forth}). That  $Z$ satisfies ({\bf Back}) can be proved in a similar way.
\end{proof}

Now we show how to construct a graded tuple bisimulation out of a graded bisimulation, with the following lemmas:

\begin{lem}\label{lem:0128}
Let $\mathcal{M}$ and $\mathcal{M'}$ be Kripke models and $Z:\mathcal{M},w\leftrightarroweq_g \mathcal{M'},w'$.
\begin{itemize}
\item[(1)] If $u\in R[w]$, then there exists $u' \in R'[w']$ with $uZu'$.
\item[(2)] If $u' \in R'[w']$, then there exists $u \in R[w]$ with $uZu'$.
\end{itemize}
\end{lem}

\begin{proof}
(1) Since $u \in R[w]$, $\{u\} \in \uparrow \mathcal{P}_{\ge 1}(R[w])$. By ({\bf Forth}), there exists $Y'\in {\uparrow \!\! \mathcal{P}_{\ge 1}(R'[w'])}$ such that $\forall y' \in Y'\exists x\in \{u\}:xZy'$. It follows that $\forall y' \in Y': uZy'$. Since  $Y'\in \uparrow\!\! \mathcal{P}_{\ge 1}(R'[w'])$, there exists $u' \in R'[w']$ such that $u' \in Y'$. It follows that $u Zu'$.

Claim (2) can be proved in a similar way by using ({\bf Back}).
\end{proof}
Let $W$ and $W'$ be sets, $X\subseteq W$, $X'\subseteq W'$ and $Z\subseteq W\times W'$. Sets $X$ and $X'$ are called a \emph{$Z$-pair} if $\forall x\in X \exists x'\in X': xZx'$ and $\forall x' \in X' \exists x\in X: xZx' $.

\begin{lem}\label{lem:0303}
Let $\mathcal{M}$ and $\mathcal{M'}$ be Kripke models and $Z:\mathcal{M},w\leftrightarroweq_g \mathcal{M'},w'$.
\begin{itemize}
\item[(1)] If $X\subseteq R[w]$ and $\abs{X}=i \ge 1$, then there exists $X'\subseteq R'[w']$ with $\abs{X'}=i$ such that $X$ and $X'$ form a $Z$-pair.
\item[(2)] If $X'\subseteq R'[w']$ and $\abs{X'}=i \ge 1$, then there exists $X\subseteq R[w]$ with $\abs{X}=i$ such that $X$ and $X'$ form a $Z$-pair.
\end{itemize}
\end{lem}

\begin{proof}
(1) The proof is by induction on $i$. If $i=1$, we may assume that $X=\{u\}$. Since $X\subseteq R[w]$, we have $u\in R[w]$. By Lemma \ref{lem:0128}, there exists $u' \in R'[w']$ with $uZu'$. Let $X'=\{u'\}$. It follows that $\abs{X'}=1$ and that $X$ and $X'$ form a $Z$-pair.

Consider the case that $i>1$. We may assume that $X=\{u\} \cup Y$, where $Y\subseteq R[w]$ and $u\not \in Y$. It follows that $\abs{Y} = i -1\ge1$. By induction hypothesis, there exists an $Y' \subseteq R'[w']$ such that $\abs{Y'}=i-1$ and that $Y$ and $Y'$ forms a $Z$-pair. Since $u\in R[w]$, by Lemma \ref{lem:0128}, there exists $u' \in R'[w']$ with $u Zu'$. If $u'\not \in Y'$, let $X'= Y' \cup \{u'\}$. Then $\abs{X'}= i$ and $X$ and $X'$ forms a $Z$-pair. 

If $u'\in Y'$, there are two subcases: $\exists y \in Y \exists v'\in {R'[w'] \! \setminus \! Y'}: yZv'$ and for all $ y \in Y$ and  $v'\in R'[w']\! \setminus \! Y'$,  not $yZv'$. 

Consider the case that $\exists y \in Y \exists v'\in {R'[w'] \! \setminus \! Y'}: yZv'$. Let $X' = Y'\cup\{v'\}$. Then $\abs{X'}= i$. Since $Y$ and $Y'$ form a $Z$-pair, $uZu'$ and $yZv'$, $X$ and $X'$ form a $Z$-pair.

Consider the case that for all $ y \in Y$ and  $v'\in R'[w']\! \setminus \! Y'$, not $yZv'$.  Since $X\in \uparrow \!\!\mathcal{P}_{\ge i}(R[w])$, by ({\bf Forth}), there exists $B'\in \uparrow\!\! \mathcal{P}_{\ge i}(R'[w'])$ such that $\forall b' \in B'\exists x\in X:xZb'$. Since  $B'\in \uparrow\!\! \mathcal{P}_{\ge i}(R'[w'])$, there exists $B'' \subseteq B'$ such that $B''\subseteq R'[w']$ and $\abs{B''} \ge i$. Since $\abs{Y'}=i-1$, there exists $b''\in B''$ such that $b''\in R'[w'] \! \setminus \! Y'$.  Since for all $ y \in Y$ and  $v'\in R'[w']\! \setminus \! Y'$, not $yZv'$, we have for all $y \in Y$, not $yZb''$. Since $\forall b' \in B'\exists x\in X:xZb'$ and $X=\{u\} \cup Y$, we have $uZb''$. Let $X' =Y'\cup\{b''\}$. Then $\abs{X'}=i$. Since $Y$ and $Y'$ form a $Z$-pair and $uZb''$, $X$ and $X'$ form a $Z$-pair.

Claim (2) can be proved in a similar way by using ({\bf Back}).
\end{proof}

\begin{prop}\label{prop:1813}
Let $\mathcal{M}=(W,R,V)$ and $\mathcal{M}'=(W',R',V')$ be Kripke models and $Z\subseteq W\times W'$ a non-empty relation such that $Z: \mathcal{M}\leftrightarroweq_{g}\mathcal{M}'$. Define a tuple of relations $\mathcal{Z}=(\mathcal{Z}_1, \mathcal{Z}_2,\ldots)$ as: $\mathcal{Z}_1 =\{(\{w\},\{w'\} )\mid  w Z w'\}$, and $\mathcal{Z}_n =\{ (X,X')\mid \abs{X}=\abs{X'}= n,~ \text{$X$ and $X'$ form a $Z$-pair} \}$, for $n >1$.
 Then $\mathcal{Z}: \mathcal{M}\leftrightarroweq_{gt}\mathcal{M}'$.
\end{prop}

\begin{proof}
Since $Z$ is non-empty, $\mathcal{Z}_1$ is non-empty. So item (1) in Definition \ref{defn:gbt} is satisfied.
Items (2), (3) and (4)  are satisfied by the definition of $Z$. Items (5) and (6) are satisfied by Lemma \ref{lem:0303}. Item (7) is satisfied by the definition of $\mathcal{Z}_i$ and the definition of $Z$-pairs.
\end{proof}

In summary, we showed how to construct a graded bisimulation out of a graded tuple bisimulation (Prop.~\ref{prop:1812}), and vice versa (Prop.~\ref{prop:1813}). Hence, graded bisimulation (Def.~\ref{defn:gradedbisi}) and graded tuple bisimulation (Def.~\ref{defn:gbt}) are equivalent. Another notion of bisimulation called \emph{resource bisimulation} was proposed in \cite{aceto2010resource}, which is very similar to the notion later proposed in \cite{ma2016dynamic}. A precise comparison of  graded bisimulation to these notions is left for future research.

\section{Conclusion}\label{sec:conclusion}
Inspired by graded models, we proposed a class of graded neighbourhood frames, and we showed that the axiomatiziation $\mbf{GrK}$ is sound and strongly complete for this class. We further showed that graded neighbourhood frames are first-order definable but not modally definable. We also obtained a new definition of graded bisimulation building upon the notion of monotonic bisimulation, where some details concerning resource bisimulation are left for further research. Our results show that techniques for monotonic modal logics can be successfully applied to graded modal logics. 

There are many options for further research:
 
(1) Using the approach developed in this paper, updating neighbourhood models \cite{ma2018update} can be compared to updated graded models \cite{ma2016dynamic}. 

(2) Building on multi-type display calculi for monotonic logics \cite{chen2019non} we plan to introduce multi-type display calculi for graded modal logic. 

(3) With yet another notion of bisimulation on graded frames, and algorithms to calculate two-sorted first-order correspondence on neighbourhood frames \cite{hansen2003monotonic,chen2019non}, we plan to get two-sorted first-order correspondence on graded frames.

(4) Finally, given the logic $\mathbf{GrK}$ in Section \ref{sec:2} for $n$ grades, and given its alternative incarnation as a monotonic modal logic in Section \ref{sec:3}, we wish to find the axiomatization of the graded modal logic for one grade. In Proposition~\ref{prop.gradmon} we showed that $(RM_n)$ is admissible in $\mbf{GrK}$. As $\mbf{GrK}$ only has necessitation for $\Box$, this is indeed of some minor interest. We can also pose this question in the other direction: is $\mbf{GrK}$ derivable in some extension of $\mbf{M}_\mbb{N}$, that makes the monotonic character of the logic clearer? Because of the axioms $(Ax4)$,  $(Ax5)$ and $(Ax6)$, we should not expect this to be without interaction axioms for different modalities. However, an interesting case is graded modal logic for a single modality $\D_n$: is there a monotonic modal logic axiomatizing this case, without interaction axioms? This logic should contain $\D_n\bot \leftrightarrow\bot$, corresponding to the requirement  that for all states $w$ in the domain of a model, $\emptyset \notin \nu_n(w)$. 
Such a logic should also contain, for example, $(\D_n \phi \wedge \D_n \neg \phi) \imp (\D_n \psi \vee \D_n \neg \psi)$. It is easy to see that this is valid in $\mbf{GrK}$. However, $(\D_n \phi \wedge \D_n \neg \phi) \imp (\D_n \psi \vee \D_n \neg \psi)$ is not derivable in monotone modal logic, as there are models of monotone modal logic in which it is false. 
\weg{Let $n$ in $\D_n$ be $2$, and consider $\mathbb M$ consisting of domain $\{a,b,c,d\}$ with $V(p)= \{a,b\}$, $V(q) = \{a,c\}$, and let $\nu(a) = \{\{a,b\}, \{c,d\}\}$ (the neighbourhoods of other points do not matter in the proof).  We now have:
\begin{itemize}
\item $\mathbb M, a \Vdash \D_2 p$ because $p$ is true in $\{a,b\}\subseteq\llbracket \neg p \rrbracket_{\mathbb{M}}$.
\item $\mathbb M, a \Vdash \D_2 \neg p$ because $\neg p$ is true in $\{c,d\}\subseteq\llbracket \neg p \rrbracket_{\mathbb{M}}$.
\item $\mathbb M, a \Vdash \neg\D_2 q$ because neither $\{a,b\}\subseteq\llbracket q \rrbracket_{\mathbb{M}} (= \{a,c\})$ nor $\{c,d\}\subseteq\{a,c\}$ (nor any of the supersets of $\{a,b\}$ and $\{c,d\}$).
\item similarly, $\mathbb M, a \Vdash \neg\D_2 \neg q$ because neither $\{a,b\}\subseteq\llbracket q \rrbracket_{\mathbb{M}} (= \{b,d\})$ nor $\{c,d\}\subseteq\{b,d\}$ (nor any of the supersets of $\{a,b\}$ and $\{c,d\}$).
\end{itemize}
}
We leave the axiomatization of single-grade graded modal logic for future research.

\bibliographystyle{BSLbibstyle}
\bibliography{ref}

\end{document}